\documentclass[amstex,12pt, amssymb]{article}

\usepackage{mathtext}
\usepackage[cp1251]{inputenc}
\usepackage[T2A]{fontenc}
\usepackage[dvips]{graphicx}
\usepackage{amsmath}
\usepackage{amssymb}
\usepackage{amsxtra}
\usepackage{latexsym}
\usepackage{ifthen}

\newcounter{lemma}[section]

\newcounter{corollary}[section]

\newcounter{remark}[section]

\newcounter{theorem}[section]

\newcounter{proposition}[section]

\newcounter{example}

\numberwithin{equation}{section}

\begin{document}

\markboth{\centerline{E.~SEVOST'YANOV}} {\centerline{ON CONTINUOUS
DISCRETE BOUNDARY EXTENSION...}}

\def\cc{\setcounter{equation}{0}
\setcounter{figure}{0}\setcounter{table}{0}}

\overfullrule=0pt


\author{EVGENY SEVOST'YANOV\\}

\title{
{\bf ON DISCRETE BOUNDARY EXTENSION OF MAPPINGS IN TERMS OF PRIME
ENDS}}

\date{\today}
\maketitle

\begin{abstract}
We study mappings that satisfy the inverse Poletsky inequality in a
domain of the Euclidean space. Under certain conditions on the
definition and mapped domains, it is established that they have a
continuous extension to the boundary in terms of prime ends if the
majorant involved in the Poletsky inequality is integrable over
spheres. Under some additional conditions, the extension mentioned
above is discrete.
\end{abstract}


\section{Introduction}

In the publication~\cite{SSD}, we studied in sufficient detail the
problem of the local and global behavior of mappings satisfying the
so-called inverse Poletsky inequality, in which the corresponding
majorant is integrable. In particular, the possibility of continuous
extension of these mappings to the boundary of the domain was shown.
In this article, we will show a little more, namely that this result
holds not only for integrable $Q,$ but also for those that have
finite integrals on spheres centered at a fixed point on a set of
radii some ''not very small'' measure. Let us point to examples of
non-integrable functions that have these finite integrals by spheres
and mappings that correspond to them (see, for example,
\cite[Examples~1,2]{SevSkv$_3$}). The key point of the manuscript is
also the discreteness of the extended mappings to the boundary in
terms of prime ends. We would also like to point out that the
manuscript is fundamentally devoted to the study of mappings in
domains with bad boundaries.

\medskip
Let us turn to the definitions. In what follows, $M_p(\Gamma)$
denotes the {\it $p$-modulus} of a family $\Gamma $
(see~\cite[Section~6]{Va}). We write $M(\Gamma)$ instead
$M_n(\Gamma).$ Let $y_0\in {\Bbb R}^n,$ $0<r_1<r_2<\infty$ and
\begin{equation}\label{eq1**}
A=A(y_0, r_1,r_2)=\left\{ y\,\in\,{\Bbb R}^n:
r_1<|y-y_0|<r_2\right\}\,.\end{equation}
Given $x_0\in{\Bbb R}^n,$ we put
$$B(x_0, r)=\{x\in {\Bbb R}^n: |x-x_0|<r\}\,, \quad {\Bbb B}^n=B(0, 1)\,,$$
$$S(x_0,r) = \{
x\,\in\,{\Bbb R}^n : |x-x_0|=r\}\,. $$
Given sets $E,$ $F\subset\overline{{\Bbb R}^n}$ and a domain
$D\subset {\Bbb R}^n$ we denote by $\Gamma(E,F,D)$ a family of all
paths $\gamma:[a,b]\rightarrow \overline{{\Bbb R}^n}$ such that
$\gamma(a)\in E,\gamma(b)\in\,F$ and $\gamma(t)\in D$ for $t \in [a,
b].$ Given a mapping $f:D\rightarrow {\Bbb R}^n,$ a point $y_0\in
\overline{f(D)}\setminus\{\infty\},$ and
$0<r_1<r_2<r_0=\sup\limits_{y\in f(D)}|y-y_0|,$ we denote by
$\Gamma_f(y_0, r_1, r_2)$ a family of all paths $\gamma$ in $D$ such
that $f(\gamma)\in \Gamma(S(y_0, r_1), S(y_0, r_2),
A(y_0,r_1,r_2)).$ Let $Q:{\Bbb R}^n\rightarrow [0, \infty]$ be a
Lebesgue measurable function. We say that {\it $f$ satisfies the
inverse Poletsky inequality at a point $y_0\in
\overline{f(D)}\setminus\{\infty\}$} if the relation
\begin{equation}\label{eq2*A}
M(\Gamma_f(y_0, r_1, r_2))\leqslant \int\limits_{A(y_0,r_1,r_2)\cap
f(D)} Q(y)\cdot \eta^{n}(|y-y_0|)\, dm(y)
\end{equation}
holds for any Lebesgue measurable function $\eta:
(r_1,r_2)\rightarrow [0,\infty ]$ such that
\begin{equation}\label{eqA2}
\int\limits_{r_1}^{r_2}\eta(r)\, dr\geqslant 1\,.
\end{equation}
Using the inversion $\psi(y)=\frac{y}{|y|^2},$ we may also define
the relation~(\ref{eq2*A}) at the point $y_0=\infty.$ A mapping $f:
D \rightarrow{\Bbb R}^n$ is called {\it discrete} if the pre-image
$\{f^{-1}\left(y\right)\}$ of any point $y\,\in\,{\Bbb R}^n$
consists of isolated points, and {\it open} if the image of any open
set $U\subset D$ is an open set in ${\Bbb R}^n.$ A mapping $f$ of
$D$ onto $D^{\,\prime}$ is called {\it closed} if $f(E)$ is closed
in $D^{\,\prime}$ for any closed set $E\subset D$ (see, e.g.,
\cite[Chapter~3]{Vu}). Let $h$ be a chordal metric in
$\overline{{\Bbb R}^n},$
$$h(x,\infty)=\frac{1}{\sqrt{1+{|x|}^2}}\,,$$
\begin{equation}\label{eq3C}
h(x,y)=\frac{|x-y|}{\sqrt{1+{|x|}^2} \sqrt{1+{|y|}^2}}\qquad x\ne
\infty\ne y\,.
\end{equation}
and let $h(E):=\sup\limits_{x,y\in E}\,h(x,y)$ be a chordal diameter
of a set~$E\subset \overline{{\Bbb R}^n}$ (see, e.g.,
\cite[Definition~12.1]{Va}).
Everywhere further the boundary $\partial A $ of the set $ A $ and
the closure $\overline{A}$ should be understood in the sense
extended Euclidean space $\overline{{\Bbb R}^n}.$ A continuous
extension of the mapping $f:D\rightarrow{\Bbb R}^n$ also should be
understood in terms of mapping with values in $\overline{{\Bbb
R}^n}$ and relative to the metric $h$ in~(\ref{eq3C}) (if a
misunderstanding is impossible). Recall that a domain $D\subset{\Bbb
R}^n$ is called {\it locally connected at the point} $x_0\in
\partial D,$ if for any neighborhood $U$ of a point $x_0$ there is a
neighborhood $V\subset U$ of $x_0$ such that $V\cap D$ is connected.
A domain $D$ is locally connected at $\partial D,$ if $D$ is locally
connected at any point $x_0\in \partial D.$ The boundary of the
domain $D$ is called {\it weakly flat} at the point $x_0\in \partial
D, $ if for any $P> 0$ and for any neighborhood $U$ of a point $x_0
$ there is a neighborhood $V\subset U$ of the same point such that
$M(\Gamma(E, F, D))> P$ for any continua $E, F \subset D,$ which
intersect $\partial U$ and $\partial V.$ The boundary of the domain
$D$ is called weakly flat if the corresponding property is fulfilled
at any point of the boundary $D.$

Recall some definitions (see, for example,~\cite{KR$_1$} and
\cite{KR$_2$}). Let $\omega$ be an open set in ${\Bbb R}^k$,
$k=1,\ldots,n-1$. A continuous mapping
$\sigma\colon\omega\rightarrow{\Bbb R}^n$ is called a {\it
$k$-dimensional surface} in ${\Bbb R}^n$. A {\it surface} is an
arbitrary $(n-1)$-dimensional surface $\sigma$ in ${\Bbb R}^n.$ A
surface $\sigma$ is called {\it a Jordan surface}, if
$\sigma(x)\ne\sigma(y)$ for $x\ne y$. In the following, we will use
$\sigma$ instead of $\sigma(\omega)\subset {\Bbb R}^n,$
$\overline{\sigma}$ instead of $\overline{\sigma(\omega)}$ and
$\partial\sigma$ instead of
$\overline{\sigma(\omega)}\setminus\sigma(\omega).$ A Jordan surface
$\sigma\colon\omega\rightarrow D$ is called a {\it cut} of $D$, if
$\sigma$ separates $D,$ that is $D\setminus \sigma$ has more than
one component, $\partial\sigma\cap D=\varnothing$ and
$\partial\sigma\cap\partial D\ne\varnothing$.

A sequence of cuts $\sigma_1,\sigma_2,\ldots,\sigma_m,\ldots$ in $D$
is called {\it a chain}, if:

(i) the set $\sigma_{m+1}$ is contained in exactly one component
$d_m$ of the set $D\setminus \sigma_m,$ wherein $\sigma_{m-1}\subset
D\setminus (\sigma_m\cup d_m)$; (ii)
$\bigcap\limits_{m=1}^{\infty}\,d_m=\varnothing.$

Two chains of cuts  $\{\sigma_m\}$ and $\{\sigma_k^{\,\prime}\}$ are
called {\it equivalent}, if for each $m=1,2,\ldots$ the domain $d_m$
contains all the domains $d_k^{\,\prime},$ except for a finite
number, and for each $k=1,2,\ldots$ the domain $d_k^{\,\prime}$ also
contains all domains $d_m,$ except for a finite number.

The {\it end} of the domain $D$ is the class of equivalent chains of
cuts in $D$. Let $K$ be the end of $D$ in ${\Bbb R}^n$, then the set
$I(K)=\bigcap\limits_{m=1}\limits^{\infty}\overline{d_m}$ is called
{\it the impression of the end} $K$. Throughout the paper,
$\Gamma(E, F, D)$ denotes the family of all paths $\gamma\colon[a,
b]\rightarrow \overline{{\Bbb R}^n}$ such that $\gamma(a)\in E,$
$\gamma(b)\in F$ and $\gamma(t)\in D$ for every $t\in[a, b].$ In
what follows, $M$ denotes the modulus of a family of paths, and the
element $dm(x)$ corresponds to the Lebesgue measure in ${\Bbb R}^n,$
$n\geqslant 2,$ see~\cite{Va}. Following~\cite{Na$_2$}, we say that
the end $K$ is {\it a prime end}, if $K$ contains a chain of cuts
$\{\sigma_m\}$ such that
$\lim\limits_{m\rightarrow\infty}M(\Gamma(C, \sigma_m, D))=0$ for
some continuum $C$ in $D.$ In the following, the following notation
is used: the set of prime ends corresponding to the domain $D,$ is
denoted by $E_D,$ and the completion of the domain $D$ by its prime
ends is denoted $\overline{D}_P.$

Consider the following definition, which goes back to
N\"akki~\cite{Na$_1$}, see also~\cite{KR$_1$}--\cite{KR$_2$}. We say
that the boundary of the domain $D$ in ${\Bbb R}^n$ is {\it locally
quasiconformal}, if each point $x_0\in\partial D$ has a neighborhood
$U$ in ${\Bbb R}^n$, which can be mapped by a quasiconformal mapping
$\varphi$ onto the unit ball ${\Bbb B}^n\subset{\Bbb R}^n$ so that
$\varphi(\partial D\cap U)$ is the intersection of ${\Bbb B}^n$ with
the coordinate hyperplane.

\medskip For a given set $E\subset {\Bbb R}^n,$ we set
$d(E):=\sup\limits_{x, y\in E}|x-y|.$
The sequence of cuts $\sigma_m,$ $m=1,2,\ldots ,$ is called {\it
regular,} if
$\overline{\sigma_m}\cap\overline{\sigma_{m+1}}=\varnothing$ for
$m\in {\Bbb N}$ and, in addition, $d(\sigma_{m})\rightarrow 0$ as
$m\rightarrow\infty.$ If the end $K$ contains at least one regular
chain, then $K$ will be called {\it regular}. We say that a bounded
domain $D$ in ${\Bbb R}^n$ is {\it regular}, if $D$ can be
quasiconformally mapped to a domain with a locally quasiconformal
boundary whose closure is a compact in ${\Bbb R}^n,$ and, besides
that, every prime end in $D$ is regular. Note that space
$\overline{D}_P=D\cup E_D$ is metric, which can be demonstrated as
follows. If $g:D_0\rightarrow D$ is a quasiconformal mapping of a
domain $D_0$ with a locally quasiconformal boundary onto some domain
$D,$ then for $x, y\in \overline{D}_P$ we put:
\begin{equation}\label{eq5}
\rho(x, y):=|g^{\,-1}(x)-g^{\,-1}(y)|\,,
\end{equation}
where the element $g^{\,-1}(x),$ $x\in E_D,$ is to be understood as
some (single) boundary point of the domain $D_0.$ The specified
boundary point is unique and well-defined by~\cite[Theorem~2.1,
Remark~2.1]{IS$_2$}, cf.~\cite[Theorem~4.1]{Na$_2$}. It is easy to
verify that~$\rho$ in~(\ref{eq5}) is a metric on $\overline{D}_P,$
and that the topology on $\overline{D}_P,$ defined by such a method,
does not depend on the choice of the map $g$ with the indicated
property.

We say that a sequence $x_m\in D,$ $m=1,2,\ldots,$ converges to a
prime end of $P\in E_D$ as $m\rightarrow\infty, $ if for any $k\in
{\Bbb N}$ all elements $x_m$ belong to $d_k$ except for a finite
number. Here $d_k$ denotes a sequence of nested domains
corresponding to the definition of the prime end $P.$ Note that for
a homeomorphism of a domain $D$ onto $D^{\,\prime}.$

\medskip
\begin{theorem}\label{th3}{\sl\, Let $D\subset {\Bbb R}^n,$
$n\geqslant 2, $ be a domain with a weakly flat boundary, and let
$D^{\,\prime}\subset {\Bbb R}^n$ be a regular domain. Suppose that
$f$ is open discrete and closed mapping of $D$ onto $D^{\,\prime}$
satisfying the relation~(\ref{eq2*A}) at any point $y_0\in
\partial D^{\,\prime}.$ Suppose that, for each point $y_0\in \partial
D^{\,\prime}$ there is $0<r_*=r_*(y_0)<\sup\limits_{y\in
D^{\,\prime}}|y-y_0|$ such that, for any $0<r_1<r_2<r_*$ there is a
set $E\subset[r_1, r_2]$ of a positive linear Lebesgue measure such
that the function $Q$ is integrable on $S(y_0, r)$ for any $r\in E.$
Then $f$ has a continuous extension
$\overline{f}:\overline{D}\rightarrow\overline{D^{\,\prime}}_P,$
while $\overline{f}(\overline{D})=\overline{D^{\,\prime}}_P.$ }
\end{theorem}

\medskip
In~\cite{Vu}, some issues related to the discreteness of a closed
quasiregular map $f:{\Bbb B}^n\rightarrow {\Bbb R}^n$ in
$\overline{{\Bbb B}^n}$ are considered, see~\cite[Lemma~4.4,
Corollary~4.5 and Theorem~4.7]{Vu}. In particular, the following
result holds (see~\cite[Theorem~4.7]{Vu}).

\medskip
{\bf Theorem.} {\sl Let $f:{\Bbb B}^n\rightarrow G^{\,\prime}$ be a
closed non-constant quasiregular mapping and let $G^{\,\prime}$ be
locally connected on the boundary. Then $f$ can be extended to a
continuous mapping $f:\overline{{\Bbb B}^n}\rightarrow
\overline{{\Bbb R}^n}$ such that $N(f)=N(\overline{f})$ and hence
$\overline{f}$ is discrete.}

\medskip
The theorem given below is devoted to a deeper study of this fact,
more precisely, we extend the mentioned result not only to a wider
class of mappings, but also to a wider class of domains. Here we
will consider the case when this extension should be understood in
terms of prime ends. Let us give some definitions.

\medskip
We say that a function ${\varphi}:D\rightarrow{\Bbb R}$ has a {\it
finite mean oscillation} at a point $x_0\in D,$ write $\varphi\in
FMO(x_0),$ if
$$\limsup\limits_{\varepsilon\rightarrow
0}\frac{1}{\Omega_n\varepsilon^n}\int\limits_{B( x_0,\,\varepsilon)}
|{\varphi}(x)-\overline{{\varphi}}_{\varepsilon}|\ dm(x)<\infty\,,
$$
where $\overline{{\varphi}}_{\varepsilon}=\frac{1}
{\Omega_n\varepsilon^n}\int\limits_{B(x_0,\,\varepsilon)}
{\varphi}(x) dm(x).$
We also say that a function ${\varphi}:D\rightarrow{\Bbb R}$ has a
finite mean oscillation at $A\subset \overline{D},$ write
${\varphi}\in FMO(A),$ if ${\varphi}$ has a finite mean oscillation
at any point $x_0\in A.$ Let
\begin{equation}\label{eq12}
q_{z_0}(r)=\frac{1}{\omega_{n-1}r^{n-1}}\int\limits_{S(z_0,
r)}Q(y)\,d\mathcal{H}^{n-1}(y)\,, \end{equation}
and $\omega_{n-1}$ denotes the area of the unit sphere ${\Bbb
S}^{n-1}$ in ${\Bbb R}^n.$  The most important statement of the
manuscript is the following.

\medskip
\begin{theorem}\label{th4}
{\sl\, Let $n\geqslant 2,$ let $D$ be a domain with a weakly flat
boundary and let $D^{\,\prime}$ be a regular domain. Let $f$ be open
discrete and closed mapping of $D$ onto $D^{\,\prime}$ for which
there is a Lebesgue measurable function $Q:{\Bbb R}^n\rightarrow [0,
\infty],$ equal to zero outside $D^{\,\prime},$ such that the
relations~(\ref{eq2*A})--(\ref{eqA2}) hold at any point $y_0\in
\partial D^{\,\prime}.$ Assume that, one of the following conditions hold:

\medskip
1) $Q\in FMO(\partial D^{\,\prime});$

\medskip
2) for any $y_0\in \partial D^{\,\prime}$ there is $\delta(y_0)>0$
such that
\begin{equation}\label{eq5F}
\int\limits_{\varepsilon}^{\delta(y_0)}
\frac{dt}{tq_{y_0}^{\frac{1}{n-1}}(t)}<\infty, \qquad
\int\limits_{0}^{\delta(y_0)}
\frac{dt}{tq_{y_0}^{\frac{1}{n-1}}(t)}=\infty
\end{equation}
for sufficiently small $\varepsilon>0.$

Then $f$ has a continuous extension
$\overline{f}:\overline{D}\rightarrow \overline{D^{\,\prime}}_P$
such that $N(f, D)=N(f, \overline{D})<\infty.$ In particular,
$\overline{f}$ is discrete in $\overline{D},$ that is,
$\overline{f}^{\,-1}(P_0)$ consists only from isolated points for
any $P_0\in E_{D^{\,\prime}}.$ }
\end{theorem}

\section{Proof of Theorem~\ref{th3}}

Let $D\subset {\Bbb R}^n,$ $f:D\rightarrow {\Bbb R}^n$ be a discrete
open mapping, $\beta: [a,\,b)\rightarrow {\Bbb R}^n$ be a path, and
$x\in\,f^{\,-1}(\beta(a)).$ A path $\alpha: [a,\,c)\rightarrow D$ is
called a {\it maximal $f$-lifting} of $\beta$ starting at $x,$ if
$(1)\quad \alpha(a)=x\,;$ $(2)\quad f\circ\alpha=\beta|_{[a,\,c)};$
$(3)$\quad for $c<c^{\prime}\leqslant b,$ there is no a path
$\alpha^{\prime}: [a,\,c^{\prime})\rightarrow D$ such that
$\alpha=\alpha^{\prime}|_{[a,\,c)}$ and $f\circ
\alpha^{\,\prime}=\beta|_{[a,\,c^{\prime})}.$ The following
statement holds (see \cite[Corollary~II.3.3]{Ri}).

\medskip
\begin{proposition}\label{pr3_a}
Let $f:D \rightarrow {\Bbb R}^n$ be a discrete open mapping, $\beta:
[a,\,b)\rightarrow {\Bbb R}^n$ be a path, and
$x\in\,f^{-1}\left(\beta(a)\right).$ Then $\beta$ has a maximal
$f$-lif\-ting starting at $x.$ If $\beta: (a,\,b]\rightarrow f(D)$
be a path, and $x\in\,f^{-1}\left(\beta(b)\right),$ then $\beta$ has
a maximal $f$-lif\-ting ending at $x.$
\end{proposition}

\medskip A path $\alpha: [a,\,b)\rightarrow D$ is called a {\it total
$f$-lifting} of $\beta$ starting at $x,$ if $(1)\quad
\alpha(a)=x\,;$ $(2)\quad (f\circ\alpha)(t)=\beta(t)$ for any $t\in
[a, b).$ In the case when the mapping $f$ is also closed, we have a
strengthened version of Proposition~\ref{pr3_a} (see, for example,
\cite[Lemma~3.7]{Vu}).

\medskip
\begin{proposition}\label{pr4_a}
Let $f:D \rightarrow {\Bbb R}^n$ be a discrete open and closed
mapping, $\beta: [a,\,b)\rightarrow f(D)$ be a path, and
$x\in\,f^{-1}\left(\beta(a)\right).$ Then $\beta$ has a total
$f$-lif\-ting starting at $x.$
\end{proposition}

\medskip
{\it Proof of Theorem~\ref{th3}.} We carry out the proof according
to a scheme similar to the proof of Theorem~1 in~\cite{Sev$_2$}. Fix
$x_0\in\partial D.$ It is necessary to show the possibility of
continuous extension of the mapping $f$ to the point $x_0.$  Using,
if necessary, the transformation $\varphi:\infty\mapsto 0$ and
taking into account the invariance of the modulus $M$ in the left
part of the relation~$\varphi:\infty\mapsto 0$
(see~\cite[Theorem~8.1]{Va}), we may assume that $x_0 \ne \infty.$

\medskip
Assume that the conclusion about the continuous extension of the
mapping $f$ to the point $x_0$ is not correct. Then any prime end
$P_0\in E_{D^{\,\prime}}$ is not a limit of $f$ at $x_0,$ in other
words, there is a sequence $x_k,$ $k=1,2,\ldots,$ $x_k\rightarrow
x_0$ as $k\rightarrow\infty$ and a number $\varepsilon_0>0$ such
that $\rho(f(x_k), P_0)\geqslant \varepsilon_0$ for any $k\in {\Bbb
N},$ where $\rho$ is one of the metrics in~(\ref{eq5}). Since
$D^{\,\prime}$ is a  regular domain by the assumption, it may be
mapped on some bounded domain $D_*$ with a locally quasiconformal
boundary using some a mapping $h: D^{\,\prime}\rightarrow D_*.$ Note
that, there is a one-to-one correspondence between boundary points
and prime ends of domains with locally quasiconformal boundaries
(see, e.g., \cite[Theorem~2.1]{IS$_2$};
cf.~\cite[Theorem~4.1]{Na$_2$}). Since $\overline{D}_*$ is a
compactum in ${\Bbb R}^n,$ we conclude from the above that a metric
space $(\overline{D^{\,\prime}}_P, \rho)$ is compact. Thus, we may
assume that $f(x_k)$ converges to some element $P_1\ne P_0,$
$P_1\in\overline{D^{\,\prime}}_P$ as $k\rightarrow\infty.$ Since, by
the assumption, $f$ has no a limit at $x_0,$ there is at least one a
sequence $y_k\rightarrow x_0$ as $k\rightarrow\infty$ such that
$\rho(f(y_k), P_1)\geqslant \varepsilon_1$ for any $k\in {\Bbb N}$
and some $\varepsilon_1>0.$ Again, sine the metric space
$(\overline{D^{\,\prime}}_P, \rho)$ is compact, we may assume that
$f(y_k)\rightarrow P_2$ as $k\rightarrow \infty,$ $P_1\ne P_2,$
$P_2\in \overline{D^{\,\prime}}_P.$ Since $f$ is closed, it
preserves the boundary of a domain, see~\cite[Theorem~3.3]{Vu}.
Thus, $P_1, P_2\in E_{D^{\,\prime}}.$

\medskip
Let $\sigma_m$ and let $\sigma^{\,\prime}_m,$ $m=0,1,2,\ldots, $ be
a sequence of cuts corresponding to prime ends $P_1$ and $P_2,$
respectively. Let also cuts $\sigma_m,$ $m=0,1,2,\ldots, $ lie on
spheres $S(z_0, r_m)$ centered at a point $z_0\in
\partial D^{\,\prime},$ where $r_m\rightarrow 0$ as $m\rightarrow\infty$
(such a sequence $\sigma_m$ exists by~\cite[Lemma~3.1]{IS$_2$},
cf.~\cite[Lemma~1]{KR$_2$}). We may assume that $r_0<r_*=r_*(z_0),$
where $r_*$ is the number from conditions of the theorem.  Let $d_m$
and $g_m,$ $m=0,1,2,\ldots, $ be sequences of domains in
$D^{\,\prime}$ corresponding to cuts $\sigma_m$ and
$\sigma^{\,\prime}_m,$ respectively. Since
$(\overline{D^{\,\prime}}_P, \rho)$ is a metric space, we may
consider that $d_m$ and $g_m$ disjoint for any $m=0,1,2,\ldots ,$ in
particular,
\begin{equation}\label{eq4}
d_0\cap g_0=\varnothing\,.
\end{equation}
Since $f(x_k)$ converges to $P_1$ as $k\rightarrow\infty,$ for any
$m\in {\Bbb N}$ there is $k=k(m)$ such that $f(x_k)\in d_m$ for
$k\geqslant k=k(m).$ By renumbering the sequence $x_k$ if necessary,
we may assume that $f(x_k)\in d_k$ for any natural $k.$ Similarly,
we may assume that $f(y_k)\in g_k$ for any $k\in {\Bbb N}.$ Fix
$f(x_1)$ and $f(y_1).$ Since, by the definition of a prime end,
$\bigcap\limits_{k=1}^{\infty}d_k=\bigcap\limits_{l=1}^{\infty}g_l=\varnothing,$
there are numbers $k_1$ and $k_2\in {\Bbb N}$ such that
$f(x_1)\not\in d_{k_1}$ and $f(y_1)\not \in g_{k_2}.$ Since, by the
definition, $d_k\subset d_{k_0}$ for any $k\geqslant k_1$ and
$g_k\subset g_{k_2}$ for $k\geqslant k_2,$ we obtain that
\begin{equation}\label{eq3}
f(x_1)\not\in d_k\,,\quad f(y_1)\not\in g_k\,, \quad
k\geqslant\max\{k_1, k_2\}\,.
\end{equation}
Let $\gamma_k$ be a path joining $f(x_1)$ and $f(x_k)$ in $d_1,$ and
let $\gamma^{\,\prime}_k$ be a path joining $f(y_1)$ and $f(y_k)$ in
$g_1.$ Let also $\alpha_k$ and $\beta_k$ be total $f$-liftings of
$\gamma_k$ and $\gamma^{\,\prime}_k$ in $D$ starting at $x_k$ and
$y_k,$ respectively (such liftings exist by
Proposition~\ref{pr4_a}).
Note that the points $f(x_1)$ and $f(y_1)$ may have no more than a
finite number of pre-images under the mapping $f$ in the domain $D,$
see~\cite[Lemma~3.2]{Vu}. Then there exists $R_0>0$ such that
$\alpha_k(1), \beta_k(1)\in D\setminus B(x_0, R_0)$ for any
$k=1,2,\ldots .$ Since the boundary of $D$ is weakly flat, for any
$P>0$ there is $i=i_P\geqslant 1$ such that
\begin{equation}\label{eq7}
M(\Gamma(|\alpha_k|, |\beta_k|, D))>P\qquad\forall\,\,k\geqslant
k_P\,.
\end{equation}
Let us to show that, the condition~(\ref{eq7}) contradicts the
definition of~$f$ in~(\ref{eq2*A}). Indeed, let $\gamma\in
\Gamma(|\alpha_k|, |\beta_k|, D).$ Then $\gamma:[0, 1]\rightarrow
D,$ $\gamma(0)\in |\alpha_k|$ and $\gamma(1)\in |\beta_k|.$ In
particular, $f(\gamma(0))\in |\gamma_k|$ and $f(\gamma(1))\in
|\gamma^{\,\prime}_k|.$ In this case, it follows from the
relations~(\ref{eq4}) and~(\ref{eq7}) that $|f(\gamma)|\cap
d_1\ne\varnothing \ne |f(\gamma)|\cap(D^{\,\prime}\setminus d_1)$
for $k\geqslant\max\{k_1, k_2\}.$ By~\cite[Theorem~1.I.5.46]{Ku}
$|f(\gamma)|\cap
\partial d_1\ne\varnothing,$ in other words, $|f(\gamma)|\cap S(z_0,
r_1)\ne\varnothing,$ because $\partial d_1\cap D^{\,\prime}\subset
\sigma_1\subset S(z_0, r_1)$ by the definition of a cut $\sigma_1.$
Let $t_1\in (0,1)$ be such that $f(\gamma(t_1))\in S(z_0, r_1)$ and
$f(\gamma)|_1:=f(\gamma)|_{[t_1, 1]}.$ Without loss of generality,
we may assume that $f(\gamma)|_1\subset {\Bbb R}^n\setminus B(z_0,
r_1).$ Arguing similarly for a path $f(\gamma)|_1,$ we may find a
point $t_2\in (t_1,1)$ such that $f(\gamma(t_2))\in S(z_0, r_0).$
Put $f(\gamma)|_2:=f(\gamma)|_{[t_1, t_2]}.$ Then $f(\gamma)|_2$ is
a subpath of $f(\gamma)$ and, in addition, $f(\gamma)|_2\in
\Gamma(S(z_0, r_1), S(z_0, r_0), D^{\,\prime}).$ Without loss of
generality, we may assume that $f(\gamma)|_2\subset B(z_0, r_0).$
Therefore, $\Gamma(|\alpha_k|, |\beta_k|, D)>\Gamma_f(z_0, r_1,
r_0).$ From the latter relation, due to the minority of the modulus
of families of paths (see e.g. \cite[Theorem~1(c)]{Fu}) we obtain
that
\begin{equation}\label{eq5C}
M(\Gamma(|\alpha_k|, |\beta_k|, D))\leqslant M(\Gamma_f(z_0, r_1,
r_0))\,.
\end{equation}
Combining~(\ref{eq5C}) with~(\ref{eq2*A}), we obtain that
\begin{equation}\label{eq11}
M(\Gamma(|\alpha_k|, |\beta_k|, D))\leqslant
\int\limits_{A(y_0,r_1,r_0)\cap f(D)} Q(y)\cdot \eta^{n}(|y-y_0|)\,
dm(y)\,,
\end{equation}
where $\eta: (r_1,r_2)\rightarrow [0,\infty ]$ is any Lebesgue
measurable function with $\int\limits_{r_1}^{r_0}\eta(r)\,
dr\geqslant 1.$

\medskip
Below we use the following conventions: $a/\infty=0$ for
$a\ne\infty,$ $a/0=\infty$ for $a>0$ and $0\cdot\infty=0$ (see,
e.g., \cite[3.I]{Sa}). Put
\begin{equation}\label{eq13}
I=\int\limits_{r_1}^{r_0}\frac{dt}{tq_{z_0}^{1/(n-1)}(t)}\,.
\end{equation}
By the assumption, there is a set $E\subset [r_1, r_0]$ of a
positive measure such that $q_{z_0}(t)$ is finite for all $t\in E.$
In this case, a function
$\eta_0(t)=\frac{1}{Itq_{z_0}^{1/(n-1)}(t)}$ satisfies the
relation~(\ref{eqA2}). Subsisting this function in the right-hand
part of~(\ref{eq11}) and using the Fubini theorem, we obtain that
\begin{equation}\label{eq14}
M(\Gamma(|\widetilde{\alpha_i}|, |\widetilde{\beta_i}|, D))
\leqslant \frac{\omega_{n-1}}{I^{n-1}}<\infty\,.
\end{equation}
The relation~(\ref{eq14}) contradicts with~(\ref{eq7}). The
contradiction obtained above disproves the assumption on the absence
of a continuous extension of the mapping $f$ to the boundary of the
domain $D.$ The proof of the equality
$\overline{f}(\overline{D})=\overline{D^{\,\prime}}$ is similar to
the second part of the proof of Theorem~3.1 in~\cite{SSD}.~$\Box$

\medskip
\begin{remark}\label{rem1}
The statement of Theorem~\ref{th3} remains true, if in its
formulation instead of the specified conditions on function $Q$ to
require that $Q\in L_{\rm loc}^1({\Bbb R}^n),$ $Q(y)\equiv 0$ for
$y\in {\Bbb R}^n\setminus f(D).$  Indeed,
by~\cite[Theorem~III.8.1]{Sa}, for any point $y_0\in{\Bbb R}^n$
\begin{equation}\label{eq7F}
\int\limits_{\varepsilon_1}^{\varepsilon_2}\int\limits_{S(y_0,
r)}Q(y)\,d\mathcal{A}\,dr\,=\int\limits_{\varepsilon_1<|y-
y_0|<\varepsilon_2}Q(y)\,dm(y)
\end{equation}
for any $0\leqslant \varepsilon_1<\varepsilon_2.$ By~(\ref{eq7F}),
it follows that $q_{y_0}(r)<\infty$ for
$\varepsilon_1<r<\varepsilon_2.$
\end{remark}

\medskip
\begin{remark}\label{rem2}
The statement of Theorem~\ref{th3} remains true, if in its
formulation instead of the specified conditions on function $Q$ to
require that, for any $y_0\in \partial D^{\,\prime}$ there is
$\delta(y_0)>0$ such that
\begin{equation}\label{eq5**}
\int\limits_{\varepsilon}^{\delta(y_0)}
\frac{dt}{tq_{y_0}^{\frac{1}{n-1}}(t)}<\infty, \qquad
\int\limits_{0}^{\delta(y_0)}
\frac{dt}{tq_{y_0}^{\frac{1}{n-1}}(t)}=\infty
\end{equation}
for sufficiently small $\varepsilon>0.$ This statement may be proved
by the choosing of the admissible function $\eta$ in~(\ref{eq11})
and by the using the fact that the second condition in~(\ref{eq5**})
is possible only if the inequality $q_{y_0}(t)<\infty$ holds for
some set $E\subset [\varepsilon, \delta(y_0)]$ of a positive linear
measure.
\end{remark}

\medskip
\begin{remark}\label{rem3}
The statement of Theorem~\ref{th3} remains true, if in its
formulation instead of the specified conditions on function $Q$ to
require that, for any $y_0\in \partial D^{\,\prime}$ there is
$\varepsilon_0=\varepsilon_0(y_0)>0$ and a Lebesgue measurable
function $\psi:(0, \varepsilon_0)\rightarrow [0, \infty]$ such that
\begin{equation}\label{eq7B} I(\varepsilon,
\varepsilon_0):=\int\limits_{\varepsilon}^{\varepsilon_0}\psi(t)\,dt
< \infty\quad \forall\,\,\varepsilon\in (0, \varepsilon_0)\,,\quad
I(\varepsilon, \varepsilon_0)>0\quad
\text{as}\quad\varepsilon\rightarrow 0\,,
\end{equation}
and, in addition,
\begin{equation} \label{eq7C}
\int\limits_{A(y_0, \varepsilon, \varepsilon_0)}
Q(x)\cdot\psi^{\,n}(|y-y_0|)\,dm(y)\leqslant C_0I^n(\varepsilon,
\varepsilon_0)\,,\end{equation}
as $\varepsilon\rightarrow 0,$ where $C_0$ is some constant, and
$A(y_0, \varepsilon, \varepsilon_0)$ is defined in~(\ref{eq1**}).

\medskip
Indeed, literally repeating the proof of the statement given in
Theorem~\ref{th3} to the ratio~(\ref{eq11}) inclusive, we put
$$\eta(t)=\left\{
\begin{array}{rr}
\psi(t)/I(r_1, r_0), & t\in (r_1, r_0)\,,\\
0,  &  t\not\in (r_1, r_0)\,,
\end{array}
\right. $$
where $I(r_1, \varepsilon_0)=\int\limits_{r_1}^{\varepsilon_0}\,\psi
(t)\, dt.$ Observe that
$\int\limits_{r_1}^{\varepsilon_0}\eta(t)\,dt=1.$ Now, by the
definition of $f$ in~(\ref{eq2*A}) and due to the
relation~(\ref{eq11}) we obtain that
\begin{equation}\label{eq14C}
M(\Gamma(|\widetilde{\alpha_i}|, |\widetilde{\beta_i}|, D))\leqslant
C_0<\infty\,.
\end{equation}
The relation~(\ref{eq14C}) contradicts with~(\ref{eq7}). The
resulting contradiction proves the desired statement.~$\Box$
\end{remark}

\section{On the discreteness of mappings with
the inverse Poletsky inequality at the boundary of a domain}

In~\cite{Vu}, some issues related to the discreteness of a closed
quasiregular map $f:{\Bbb B}^n\rightarrow {\Bbb R}^n$ in
$\overline{{\Bbb B}^n}$ are considered, see~\cite[Lemma~4.4,
Corollary~4.5 and Theorem~4.7]{Vu}. In this section we talk about
the discreteness of mappings that satisfy the
condition~(\ref{eq2*A}). Among other things, we note that we are
primarily interested here in the case when the mapped domain has a
bad boundary.

\medskip
We will say that {\it $f$ satisfies the inverse Poletsky inequality}
at a point $y_0\in\overline{f(D)}\setminus \{\infty\}$ relative to
$p$-modulus, if the relation

\begin{equation}\label{eq2*B}
M_p(\Gamma_f(y_0, r_1, r_2))\leqslant
\int\limits_{A(y_0,r_1,r_2)\cap f(D)} Q(y)\cdot \eta^p (|y-y_0|)\,
dm(y)
\end{equation}
holds for any Lebesgue measurable function $\eta:
(r_1,r_2)\rightarrow [0,\infty ]$ such that
\begin{equation}\label{eqB2}
\int\limits_{r_1}^{r_2}\eta(r)\, dr\geqslant 1\,.
\end{equation}
Using the inversion $\psi(y)=\frac{y}{|y|^2},$ we also may defined
the relation~(\ref{eq2*B}) at the point $y_0=\infty.$

Following~\cite[Section~2.4]{NP}, we say that a domain $D\subset
{\Bbb R}^n,$ $n\geqslant 2,$ is {\it uniform with respect to
$p$-modulus}, if for any $r>0$ there is $\delta>0$ such that the
inequality
\begin{equation}\label{eq17***}
M_p(\Gamma(F^{\,*},F, D))\geqslant \delta
\end{equation}
holds for any continua $F, F^*\subset D$ with $h(F)\geqslant r$ and
$h(F^{\,*})\geqslant r.$ When $p=n,$ the prefix ''relative to $p
$-modulus'' is omitted. Note that this is the definition slightly
different from the ''classical'' given in \cite[Chapter~2.4]{NP},
where the sets $F$ and $F^*\subset D $ are assumed to be arbitrary
connected. We prove the following statement (see its analogue for
quasiregular mappings of the unit ball in~\cite[Lemma~4.4]{Vu}).

\medskip
\begin{lemma}\label{lem1A}
{\sl\, Let $n\geqslant 2,$ $n-1<p\leqslant n,$ let $D$ be a domain
which is uniform with respect to $p$-modulus, and let $D^{\,\prime}$
be a regular domain. Let $f:D\rightarrow {\Bbb R}^n$ be an open
discrete and closed mapping in $D,$ for which there is a Lebesgue
measurable function $Q:{\Bbb R}^n\rightarrow [0, \infty],$ equals to
zero outside of $D^{\,\prime},$ such that the
relations~(\ref{eq2*B})--(\ref{eqB2}) hold for any $y_0\in
\partial D^{\,\prime}.$ Assume that, for any $y_0\in
\partial D^{\,\prime}$ there is $\varepsilon_0=\varepsilon_0(y_0)>0$ and a Lebesgue measurable
function $\psi:(0, \varepsilon_0)\rightarrow [0,\infty]$ such that
\begin{equation}\label{eq7***} I(\varepsilon,
\varepsilon_0):=\int\limits_{\varepsilon}^{\varepsilon_0}\psi(t)\,dt
< \infty\quad \forall\,\,\varepsilon\in (0, \varepsilon_0)\,,\quad
I(\varepsilon, \varepsilon_0)\rightarrow
\infty\quad\text{as}\quad\varepsilon\rightarrow 0\,,
\end{equation}
and, in addition,
\begin{equation} \label{eq3.7.2}
\int\limits_{A(y_0, \varepsilon, \varepsilon_0)}
Q(y)\cdot\psi^{\,p}(|y-y_0|)\,dm(y) = o(I^p(\varepsilon,
\varepsilon_0))\,,\end{equation}
as $\varepsilon\rightarrow 0,$ where $A(y_0, \varepsilon,
\varepsilon_0)$ is defined in~(\ref{eq1**}). Let $C_j,$
$j=1,2,\ldots ,$ be a sequence of continua such that
$h(C_j)\geqslant \delta>0$ for some $\delta>0$ and any $j\in {\Bbb
N}$ and, in addition, $\rho(f(C_j))\rightarrow 0$ as
$j\rightarrow\infty.$ Then there is $\delta_1>0$ such that
$\rho(f(C_j), P_0)\geqslant \delta_1>0$ for any $j\in {\Bbb N}$ and
for any $P_0\in E_{D^{\,\prime}},$ where the metrics $\rho$ is
defined in~(\ref{eq5}).

Here, as usually, $$\rho(A)=\sup\limits_{x, y\in A}\rho(x, y)\,,$$
$$\rho(A, B)=\inf\limits_{x\in A, y\in B}\rho(x, y)\,.$$
}
\end{lemma}

\begin{proof}
Suppose the opposite, namely, let $\rho(f(C_{j_k}), P_0)\rightarrow
0$ as $k\rightarrow\infty$ for some $P_0\in E_{D^{\,\prime}}$ and
for some increasing sequence of numbers $j_k,$ $k=1,2,\ldots .$ Let
$F\subset D$ be any continuum in $D,$ and let $\Gamma_k:=\Gamma(F,
C_{j_k}, D).$ Due to the definition of the uniformity of the domain
with respect to $p$-modulus, we obtain that
\begin{equation}\label{eq3A}
M_p(\Gamma_k)\geqslant \delta_2>0
\end{equation}
for any $k\in {\Bbb N}$ and some $\delta_2>0.$ On the other hand,
let us to consider the family of paths~$f(\Gamma_k).$ Let $d_l,$
$l=1,2,\ldots ,$ be a sequence of domains which corresponds to the
prime end $P_0,$ and let $\sigma_l$ be a cut corresponding to $d_l.$
We may assume that $\sigma_l,$ $l=1,2,\ldots, $ lie on spheres
$S(y_0, r_l)$ centered at some point $y_0\in
\partial D^{\,\prime},$ where $r_l\rightarrow 0$ as $l\rightarrow\infty$
(see~\cite[Lemma~3.1]{IS$_2$}, cf.~\cite[Lemma~1]{KR$_2$}).

\medskip
Let us to prove that, for any $l\in {\Bbb N}$ there is a number
$k=k_l$ such that
\begin{equation}\label{eq3M}
f(C_{j_k})\subset d_l\,,\qquad k\geqslant k_l\,.
\end{equation}
Suppose the opposite. Then there is $l_0\in {\Bbb N}$ such that
\begin{equation}\label{eq3F}
f(C_{j_{m_l}})\cap ({\Bbb R}^n\setminus d_{l_0})\ne\varnothing
\end{equation}
for some increasing sequence of numbers  $m_l,$ $l=1,2,\ldots .$  In
this case, there is a sequence $x_{m_l}\in f(C_{j_{m_l}})\cap ({\Bbb
R}^n\setminus d_{l_0}),$ $l\in {\Bbb N}.$ Since by the assumption
$\rho(f(C_{j_k}), P_0)\rightarrow 0$ for some sequence of numbers
$j_k,$ $k=1,2,\ldots ,$ we obtain that
\begin{equation}\label{eq3E}
\rho(f(C_{j_{m_l}}), P_0)\rightarrow 0\qquad {\text as}\qquad
l\rightarrow\infty\,.
\end{equation}
Since $\rho(f(C_{j_{m_l}}), P_0)=\inf\limits_{y\in
f(C_{j_{m_l}})}h(y, P_0)$ and $f(C_{j_{m_l}})$ is a compact set in
$\overline{D^{\,\prime}}_P$ as a continuous image of the compactum
$C_{j_{m_l}}$ under the mapping $f,$ it follows that
$\rho(f(C_{j_{m_l}}), P_0)=\rho(y_l, P_0),$ where $y_l\in
f(C_{j_{m_l}}).$ Due to the relation~(\ref{eq3E}) we obtain that
$y_l\rightarrow y_0$ as $l\rightarrow\infty$ in the metric $\rho.$
Since by the assumption $\rho(f(C_j))=\sup\limits_{y,z\in
f(C_j)}\rho(y,z)\rightarrow 0$ as $j\rightarrow\infty,$ we have that
$\rho(y_l, x_{m_l})\leqslant \rho(f(C_{j_{m_l}}))\rightarrow 0$ as
$l\rightarrow\infty.$  Now, by the triangle inequality, we obtain
that
$$\rho(x_{m_l}, P_0)\leqslant \rho(x_{m_l}, y_l)+\rho(y_l, P_0)
\rightarrow 0\qquad {\text as}\quad l\rightarrow\infty\,.$$
The latter contradicts with~(\ref{eq3F}). The contradiction obtained
above proves~(\ref{eq3M}).

\medskip
The following considerations are similar to the second part of the
proof of Lemma~2.1 in~\cite{Sev$_1$}. Without loss of generality we
may consider that the number $l_0\in {\Bbb N}$ is such that
$r_l<\varepsilon_0$ for any $l\geqslant l_0,$ and
\begin{equation}\label{eq3I}
f(F)\subset {\Bbb R}^n\setminus d_1\,.
\end{equation}
In this case, we observe that, for $l\geqslant 2$
\begin{equation}\label{eq3G}
f(\Gamma_{k_l})>\Gamma(S(y_0, r_l), S(y_0, r_1), A(y_0, r_l,
r_1))\,.
\end{equation}
Indeed, let $\widetilde{\gamma}\in f(\Gamma_{k_l}).$ Then
$\widetilde{\gamma}(t)=f(\gamma(t)),$ where $\gamma\in
\Gamma_{k_l},$ $\gamma:[0, 1]\rightarrow D,$ $\gamma(0)\in F,$
$\gamma(1)\in C_{j_{k_l}}.$ Due to the relation~(\ref{eq3I}), we
obtain that $f(\gamma(0))\in f(F)\subset {\Bbb R}^n\setminus B(y_0,
\varepsilon_0).$ On the other hand, by~(\ref{eq3M}), $\gamma(1)\in
C_{j_{k_l}}\subset d_l\subset d_1.$ Thus, $|f(\gamma(t))|\cap
d_1)\ne\varnothing \ne |f(\gamma(t))|\cap ({\Bbb R}^n\setminus
d_1).$ Now, by~\cite[Theorem~1.I.5.46]{Ku} we obtain that, there is
$0<t_1<1$ such that $f(\gamma(t_1))\in \partial d_1\cap D\subset
S(y_0, r_1).$ Set $\gamma_1:=\gamma|_{[t_1, 1]}.$ We may consider
that $f(\gamma(t))\in d_1$ for any $t\geqslant t_1.$ Arguing
similarly, we obtain $t_2\in [t_1, 1]$ such that $f(\gamma(t_2))\in
S(y_0, r_l).$ Put $\gamma_2:=\gamma|_{[t_1, t_2]}.$ We may consider
that $f(\gamma(t))\in d_l$ for any $t\in [t_1, t_2].$ Now, a path
$f(\gamma_2)$ is a subpath of $f(\gamma)=\widetilde{\gamma},$ which
belongs to $\Gamma(S(y_0, r_l), S(y_0, r_1), A(y_0,r_l, r_1)).$ The
relation~(\ref{eq3G}) is established.

\medskip
It follows from~(\ref{eq3G}) that
\begin{equation}\label{eq3H}
\Gamma_{k_l}>\Gamma_{f}(S(y_0, r_l), S(y_0, r_1), A(y_0, r_l,
r_1))\,.
\end{equation}
Set
$$\eta_{l}(t)=\left\{
\begin{array}{rr}
\psi(t)/I(r_l, r_1), & t\in (r_l, r_1)\,,\\
0,  &  t\not\in (r_l, r_1)\,,
\end{array}
\right. $$
where $I(r_l, r_1)=\int\limits_{r_l}^{r_1}\,\psi (t)\, dt.$ Observe
that
$\int\limits_{r_l}^{r_1}\eta_{l}(t)\,dt=1.$ Now, by the
relations~(\ref{eq3.7.2}) and~(\ref{eq3H}), and due to the
definition of $f$ in~(\ref{eq2*B}), we obtain that
$$M_p(\Gamma_{k_l})\leqslant M_p(\Gamma_{f}(S(y_0, r_l), S(y_0,
r_1), A(y_0, r_l, r_1)))\leqslant$$
\begin{equation}\label{eq3J}
\leqslant \frac{1}{I^p(r_l, r_1)}\int\limits_{A(y_0, r_l, r_1)}
Q(y)\cdot\psi^{\,p}(|y-y_0|)\,dm(y)\rightarrow 0\quad \text{as}\quad
l\rightarrow\infty\,.
\end{equation}
The relation~(\ref{eq3J}) contradicts with~(\ref{eq3A}). The
contradiction obtained above proves the lemma.~$\Box$
\end{proof}

\medskip
\begin{corollary}
{\sl\, The statement of Lemma~\ref{lem1A} is fulfilled if we put
$D={\Bbb B}^n.$ }
\end{corollary}

\medskip
\begin{proof}
Obviously, the domain $D={\Bbb B}^n$ is locally connected at its
boundary. We prove that this domain is uniform with respect to the
$p$-modulus for $p\in (n-1, n).$ Indeed, since ${\Bbb B}^n$ is a
Loewner space (see~\cite[Example~8.24(a)]{He}), the set ${\Bbb B}^n$
is Ahlfors regular with respect to the Euclidean metric $d$ and
Lebesgue measure in ${\Bbb R}^n$  (see~\cite[Proposition~8.19]{He}).
In addition, in ${\Bbb B}^n,$ $(1; p)$-Poincar\'{e} inequality holds
for any $p\geqslant 1$ (see e.g.~\cite[Theorem~10.5]{HaK}). Now,
by~\cite[Proposition~4.7]{AS} we obtain that the relation
\begin{equation}\label{eq1H}
M_p(\Gamma(E, F, {\Bbb B}^n))\geqslant \frac{1}{C}\min\{{\rm
diam}\,E, {\rm diam}\,F\}\,,
\end{equation}
holds for any $n-1<p\leqslant n$ and for any continua $E, F\subset
{\Bbb B}^n,$ where $C>0$ is some constant, and ${\rm diam}$ denotes
the Euclidean diameter. Since the Euclidean distance is equivalent
to the chordal distance on bounded sets, the uniformity of the
domain $D={\Bbb B}^n$ with respect to the $p$-modulus follows
directly from~(\ref{eq1H}).~$\Box$
\end{proof}

\medskip
We need the following statement (see~\cite[Theorem~4.2]{Na$_1$}).

\medskip
\begin{proposition}\label{pr1}
{\sl\, Let $\frak{F}$ be a family of connected sets in $D$ such that
$\inf\limits_{F\subset \frak{F}}h(F)>0,$ and let
$\inf\limits_{F\subset \frak{F}}M(\Gamma(F, A, D))>0$ for some
continuum $A\subset D.$ Then
$$\inf\limits_{F, F^{\,*}\subset
\frak{F}}M(\Gamma(F, F^{\,*}, D))>0\,.$$ }
\end{proposition}
Let $p\geqslant 1.$ Due to~\cite[Section~3]{MRSY} we say that a
boundary $D$ is called {\it strongly accessible with respect to
$p$-modulus at $x_0\in \partial D,$} if for any neighborhood $U$ of
the point $x_0\in\partial D$ there is a neighborhood $V\subset U$ of
this point, a compactum $F\subset D$ and a number $\delta>0$ such
that $M_p(\Gamma(E, F, D))\geqslant \delta$ for any continua
$E\subset D$ such that $E\cap \partial U\ne\varnothing\ne E\cap
\partial V.$ The boundary of a domain $D$ is called {\it strongly accessible with respect to
$p$-modulus,} if this is true for any $x_0\in
\partial D.$ When $p=n,$ prefix ''relative to $p$-modulus'' is omitted.
The following lemma is valid (see the statement similar in content
to~\cite[Theorem~6.2]{Na$_1$}).

\medskip
\begin{lemma}\label{lem4A}
{\sl\, A domain $D\subset {\Bbb R}^n$ has a strongly accessible
boundary if and only if $D$ is uniform.}
\end{lemma}

\medskip
\begin{proof}
The fact that uniform domains have strongly accessible boundaries
has been proved in~\cite[Remark~1]{SevSkv$_1$}. It remains to prove
that domains with strongly accessible boundaries are uniform.

We will prove this statement from the opposite. Let $D$ be a domain
which has a strongly accessible boundary, but it is not uniform.
Then there is $r>0$ such that, for any $k\in {\Bbb N}$ there are
continua $F_k$ and $F^{\,*}_k\subset D$ such that $h(F_k)\geqslant
r,$ $h(F^{\,*}_k)\geqslant r,$ however,
\begin{equation}\label{eq4A}
M(\Gamma(F_k, F^*_k, D))<1/k\,.
\end{equation}
Let $x_k\in F_k.$ Since $\overline{D}$ is compact in
$\overline{{\Bbb R}^n},$ we may assume that $x_k\rightarrow x_0\in
\overline{D}.$ Note that the strongly accessibility of the domain
$D$ at the boundary points is assumed to be, and at the inner points
it is even weakly flat, which is the result of V\"{a}is\"{a}l\"{a}'s
lemma (see e.g.~\cite[Sect.~10.12]{Va}, cf.
\cite[Lemma~2.2]{SevSkv$_2$}). Let $U$ be a neighborhood of the
point $x_0$ such that $h(x_0,
\partial U)\leqslant r/2.$ Then there is a neighborhood $V\subset U,$ a compactum $F\subset D$
and a number $\delta>0$ such that the relation $M(\Gamma(E, F,
D))\geqslant \delta$ holds for any continuum $E\subset D$ such that
$E\cap
\partial U\ne\varnothing\ne E\cap
\partial V.$ By the choice of the neighborhood $U,$
we obtain that $F_k\cap U\ne\varnothing\ne F_k\cap (D\setminus U)$
for sufficiently large $k\in {\Bbb N}.$ Observe that, for the same
$k\in {\Bbb N},$ the condition $F_k\cap V\ne\varnothing\ne F_k\cap
(D\setminus V)$ holds. Then, by~\cite[Theorem~1.I.5.46]{Ku} we
obtain that $F_k\cap
\partial U\ne\varnothing\ne F_k\cap \partial V.$ Observe that, a compactum
$F$ can be imbedded in some continuum $A\subset D$
(see~\cite[Lemma~1]{Sm}). Then the inequality $M(\Gamma(E, A,
D))\geqslant \delta$ will only increase. Given the above, we obtain
that
\begin{equation}\label{eq4E}
M(\Gamma(F_k, A, D))\geqslant \delta\qquad \forall\,\,k\geqslant k_0
\end{equation}
for some. Taking~$\inf$ over all $k\geqslant k_0$ in~(\ref{eq4E}),
we obtain that
\begin{equation}\label{eq4F}
\inf\limits_{k\geqslant k_0} M(\Gamma(F_k, A, D))\geqslant \delta\,.
\end{equation}
Set $\frak{F}:=\left\{F_k\right\}_{k=k_0}^{\infty}.$ Now, by the
condition~(\ref{eq4F}) and by Proposition~\ref{pr1}, we obtain that
$\inf\limits_{k\geqslant k_0} M(\Gamma(F_k, F^{\,*}_k, D))>0,$ that
contradicts the assumption made in~(\ref{eq4A}). The resulting
contradiction completes the proof of the lemma.~$\Box$
\end{proof}

\medskip
Obviously, weakly flat boundaries are strongly accessible. Now, by
Lemma~\ref{lem4A} we obtain the following.

\medskip
\begin{corollary}\label{cor1}
{\sl\, If $D\subset {\Bbb R}^n$ has a weakly flat boundary, then $D$
is uniform.}
\end{corollary}

\medskip
Given a mapping $f:D\,\rightarrow\,{\Bbb R}^n,$ a set $E\subset D$
and $y\,\in\,{\Bbb R}^n,$ we define the {\it multiplicity function
$N(y,f,E)$} as a number of preimages of the point $y$ in a set $E,$
i.e.
$$
N(y,f,E)\,=\,{\rm card}\,\left\{x\in E: f(x)=y\right\}\,,
$$
\begin{equation}\label{eq1G}
N(f,E)\,=\,\sup\limits_{y\in{\Bbb R}^n}\,N(y,f,E)\,.
\end{equation}
Note that, the concept of a multiplicity function may also be
extended to sets belonging to the closure of a given domain.
Finally, we formulate and prove a key statement about the
discreteness of mapping (see~\cite[Theorem~4.7]{Vu}).

\medskip
\begin{lemma}\label{lem3A}
{\sl\, Suppose that, $p=n,$ the domain $D$ is weakly flat and the
domain $D^{\,\prime}$ is regular. Assume that, there is a Lebesgue
measurable function $Q:{\Bbb R}^n\rightarrow [0, \infty],$ equals to
zero outside of $D^{\,\prime},$ such that the
relations~(\ref{eq2*B})--(\ref{eqB2}) hold for any $y_0\in
\partial D^{\,\prime}.$ Assume that, for any $y_0\in
\partial D^{\,\prime}$ there is $\varepsilon_0=\varepsilon_0(y_0)>0$ and a Lebesgue measurable
function $\psi:(0, \varepsilon_0)\rightarrow [0,\infty]$ such
that~(\ref{eq7***})--(\ref{eq3.7.2}) as $\varepsilon\rightarrow 0,$
where $A(y_0, \varepsilon, \varepsilon_0)$ is defined
in~(\ref{eq1**}).

Then the mapping $f$ has a continuous extension
$\overline{f}:\overline{D}\rightarrow \overline{D^{\,\prime}}_P$
such that $N(f, D)=N(f, \overline {D})<\infty.$ In particular,
$\overline{f}$ is discrete in $\overline{D},$ that is,
$\overline{f}^{\,-1}(P_0)$ consists only from isolated points for
any $P_0\subset E_{D^{\,\prime}}.$ }
\end{lemma}

\medskip
\begin{proof}
First of all, the possibility of continuous extension of $f$ to a
mapping $\overline{f}:\overline{D}\rightarrow
\overline{D^{\,\prime}}_P$ follows by Remark~\ref{rem3}. Note also
that $N(f, D) <\infty,$ see~\cite[Theorem~2.8]{MS}. Let us to prove
that $N(f, D)=N(f, \overline{D}).$ Next we will reason using the
scheme proof of Theorem~4.7 in \cite{Vu}. Assume the contrary. Then
there are points $P_0\in E_{D^{\,\prime}}$ and $x_1,x_2,\ldots, x_k,
x_{k+1}\in
\partial D$ such that $f(x_i)=P_0,$ $i=1,2,\ldots, k+1$ and $k:=N(f,
D).$ Since by the assumption $D^{\,\prime}$ is regular, there is a
mapping $g$ of some domain with a locally quasiconformal boundary
$D_0$ onto $D^{\,\prime}.$ Let us consider the mapping $F:=f\circ
g^{\,-1}.$ Note that, by the definition of a domain with a locally
quasiconformal boundary, $D_0$ is locally connected on $\partial
D_0.$ Note that, the mapping $F$ has a continuous extension
$\overline{F}:\overline{D}\rightarrow \overline{D_0}$ to
$\overline{D},$ and $\overline{F}(\overline{D})=\overline{D_0}$
(this follows from the fact that each of the mappings $f$ and
$g^{\,-1}$ has a continuous extension to $\overline{D}$ and
$\overline{D^{\,\prime}}_P,$ respectively). Set
$y_0:=\overline{F}(P_0)\in D_0.$ Now, for any $p\in {\Bbb N}$ there
is a neighborhood $\widetilde{U^{\,\prime}_p}\subset B(y_0, 1/p)$ of
$y_0$ such that the set $\widetilde{U^{\,\prime}_p}\cap
D_0=U^{\,\prime}_p$ is connected.

\medskip
Let us to prove that, for any $i=1,2,\ldots, k+1$ there is a
component $V_p^i$ of the set $F^{\,-1}(U^{\,\prime}_p)$ such that
$x_i\in\overline{V_p^i}.$ Fix $i=1,2,\ldots, k+1.$ By the continuity
of $F$ in $\overline{D},$ there is $r_i=r_i(x_i)>0$ such that
$f(B(x_i, r_i)\cap D)\subset U^{\,\prime}_p.$
By~\cite[Lemma~3.15]{MRSY}, a domain with a weakly flat boundary is
locally connected on its boundary. Thus, we may find a neighborhood
$W_i\subset B(x_i, r_i)$ of the point $x_i$ such that $W_i\cap D$ is
connected. Then $W_i\cap D$ belongs to one and only one component
$V^p_i$ of the set $F^{\,-1}(U^{\,\prime}_p),$ while
$x_i\in\overline{W_i\cap D}\subset \overline{V_p^i},$ as required.

Next we show that the sets $\overline{V_p^i}$ are disjoint for any
$i=1,2,\ldots, k+1$ and large enough $p\in {\Bbb N}.$ In turn, we
prove for this that $h(\overline{V_p^i})\rightarrow 0$ as
$p\rightarrow\infty$ for each fixed $i=1,2,\ldots, k+1.$ Let us
prove the opposite. Then there is $1\leqslant i_0\leqslant k+1,$ a
number $r_0>0,$ $r_0<\frac{1}{2}\min\limits_{1\leqslant i,
j\leqslant k+1, i\ne j}h(x_i, x_j)$ and an increasing sequence of
numbers $p_m,$ $m=1,2,\ldots,$ such that $S_h(x_{i_0}, r_0)\cap
\overline{V_{p_m}^{i_0}}\ne\varnothing,$ where $S_h(x_0, r)=\{x\in
\overline{{\Bbb R}^n}: h(x, x_0)=r\},$ and $h$ denotes the chordal
metric in $\overline{{\Bbb R}^n}.$ In this case, there are $a_m,
b_m\in V_{p_m}^{i_0}$ such that $a_m\rightarrow x_{i_0}$ as
$m\rightarrow\infty$ and $h(a_m, b_m)\geqslant r_0/2.$ Join the
points $a_m$ and $b_m$ by a path $C_m,$ which entirely belongs to
$V_{p_m}^{i_0}.$ Then $h(|C_m|)\geqslant r_0/2$ for $m=1,2,\ldots .$
On the other hand, since $|C_m|\subset f(V_{p_m}^{i_0})\subset
B(y_0, 1/p_m),$ then simultaneously $h(F(|C_m|))\rightarrow 0$ as
$m\rightarrow\infty$ and $h(F(|C_m|), y_0)\rightarrow 0$ as
$m\rightarrow\infty.$ Now, by the definition of the metric $\rho$
in~(\ref{eq5}) and of the mapping $g,$ we obtain that
$\rho(f(|C_m|))\rightarrow 0$ as $m\rightarrow\infty$ and
$\rho(f(|C_m|), y_0)\rightarrow 0$ as $m\rightarrow\infty,$ that
contradicts with Lemma~\ref{lem1A}. The resulting contradiction
indicates the incorrectness of the above assumption.

By~\cite[Lemma~3.6]{Vu} $F$ is a mapping of $\overline{V_p^i}$ onto
$U^{\,\prime}_p$ for any $i=1,2,\ldots, k, k+1.$ Thus, $N(f, D)=N(F,
D)\geqslant k+1,$ which contradicts the definition of the number
$k.$ The obtained contradiction refutes the assumption that $N(f,
\overline{D})>N(f, D).$ The lemma is proved.~$\Box$
\end{proof}

\medskip
Let us now turn to the main results of this section.

\medskip
{\it Proof of Theorem~\ref{th4}}. In the case  1), we choose
$\psi(t)=\frac{1}{t\log\frac{1}{t}},$ and in the case 2), we set
$$\psi(t)\quad=\quad \left \{\begin{array}{rr}
1/[tq^{\frac{1}{n-1}}_{y_0}(t)]\ , & \ t\in (\varepsilon,
\varepsilon_0)\ ,
\\ 0\ ,  &  \ t\notin (\varepsilon,
\varepsilon_0)\ ,
\end{array} \right.$$
Observe that, the relations~(\ref{eq7***})--(\ref{eq3.7.2}) hold for
these functions $\psi,$ where $p=n$ (the proof of this facts may be
found in~\cite[Proof of Theorem~1.1]{Sev$_1$}). The desired
conclusion follows from Lemma~\ref{lem3A}.~$\Box$


\medskip
{\bf \noindent Evgeny Sevost'yanov} \\
{\bf 1.} Zhytomyr Ivan Franko State University,  \\
40 Bol'shaya Berdichevskaya Str., 10 008  Zhytomyr, UKRAINE \\
{\bf 2.} Institute of Applied Mathematics and Mechanics\\
of NAS of Ukraine, \\
1 Dobrovol'skogo Str., 84 100 Slavyansk,  UKRAINE\\
esevostyanov2009@gmail.com

\end{document}